\documentclass[11pt]{amsart}

\usepackage{nicematrix}
\usepackage{amsmath,amsthm,amssymb,hyperref,enumerate,xcolor,txfonts}
\usepackage{amsaddr}
\usepackage[all]{xy}
\usepackage[left=2cm,right=2cm,top=3cm,bottom=3cm]{geometry}

\usepackage{scrextend}
\changefontsizes{12pt}

\newtheorem{thm}{Theorem}[section]
\newtheorem{cor}[thm]{Corollary}
\newtheorem{lem}[thm]{Lemma}
\newtheorem{prop}[thm]{Proposition}
\newtheorem{conj}[thm]{Conjecture}
\newtheorem{definition}[thm]{Definition}
\theoremstyle{remark}

\newtheorem{rmk}[thm]{Remark}

\newtheorem{example}[thm]{Example}

\def\F{\mathbb F}

\def\GL{\mathrm{GL}}
\def\SL{\mathrm{SL}}

\def\Bcal{\mathcal{B}}

\def\Qcal{\mathcal{Q}}
\def\Tcal{\mathcal{T}}
\def\P{\mathbf{P}}

\def\Par{\mathrm{P}}
\def\B{\mathrm{B}}

\newcommand{\pmt}[1]{\begin{bmatrix}#1\end{bmatrix}}
\newcommand{\bmt}[1]{\begin{bmatrix}#1\end{bmatrix}}

\begin{document}

	\title[A proof of the Lewis-Reiner-Stanton conjecture]{A proof of the Lewis-Reiner-Stanton conjecture for the Borel subgroup}
  \author[Ha]{Le Minh Ha}
  \address{Faculty of Mathematics-Mechanics-Informatics, University of Science, Vietnam National University, Hanoi}
  \author[Hai]{Nguyen Dang Ho Hai}
  \address{Department of Mathematics, College of Sciences, University of Hue, Vietnam}
 \author[Nghia]{Nguyen Van Nghia}
  \address{Department of Natural Sciences, Hung Vuong University, Phu Tho, Vietnam}
  \email{leminhha@hus.edu.vn,ndhhai@husc.edu.vn,nguyenvannghia@hvu.edu.vn}

	\begin{abstract}
For each parabolic subgroup $\Par$ of the general linear group $\GL_n(\F_q)$, a conjecture due to Lewis, Reiner and Stanton \cite{LewisReinerStanton2017} predicts a formula for the Hilbert series of the space of invariants $\Qcal_m(n)^\Par$ where $\Qcal_m(n)$ is the quotient ring $\F_q[x_1,\ldots,x_n]/(x_1^{q^m},\ldots,x_n^{q^m})$. In this paper, we prove the conjecture for the Borel subgroup $\B$ by constructing a linear basis for $\Qcal_m(n)^\B$. The construction is based on an operator $\delta$ which produces new invariants from old invariants of lower ranks. We also upgrade the conjecture of Lewis, Reiner and Stanton by proposing an explicit basis for the space of invariants for each parabolic subgroup.
	\end{abstract}

	\maketitle
	
	\tableofcontents

	\section{Introduction and statement of the main result}
	It is now evident that considering the symmetric group on $n$ letters as the limit when $q \to 1$ of the general linear groups $\GL_n (\mathbb{F}_q)$ over the field of $q$ elements is not merely an intriguing thought experiment but, in fact, an exceptionally useful generalization. Moreover, the hidden symmetry of many combinatorial objects can be revealed and investigated in connection with representation theory and invariant theory. 
	
	In a 2004 paper, Reiner, Stanton and White \cite{ReinerStantonWhite2004} introduced the so-called \emph{cyclic sieving phenomenon} as a general framework to organize the hidden symmetry of finite sets with group action encoded by their generating functions. There is a growing number of instances of this phenomenon in literature, which occur in many different contexts and often with a deep connection to representation theory. 
	  
In their investigation of the cyclic sieving phenomena in the theory of real reflection groups, and analogues for the $q$-Catalan and $q$-Fuss Catalan numbers in the representation theory of rational Cherednik algebras for Coxeter and complex reflection groups, Lewis, Reiner and Stanton \cite{LewisReinerStanton2017} proposed several amazing conjectures about the Hilbert series of the invariant subspaces of the truncated polynomial algebra under the action of parabolic subgroups. The purpose of this paper is to prove a version of their conjecture, the so-called parabolic conjecture, for arguably one of the most critical cases - the Borel subgroup $\B_n$ of $\GL_n (\mathbb{F}_q)$.     
	
	To state the conjectures and our main result, let us fix a positive integer $n$, a prime number $p$, and consider the polynomial algebra $\F_q[x_1,\ldots,x_n]$ over the finite field $\F_q$ of $q$ elements, where $q$ is a power of $p$. The indeterminates $x_i$ are given degree $1$, making this polynomial algebra into a graded ring. The general linear group $\GL_n:=\GL_n(\F_q)$ acts on $\F_q[x_1,\ldots,x_n]$ by linear substitutions. More precisely, if $\sigma=(\sigma_{ij})\in \GL_n$, then 
	\[
	\sigma f(x_1,\ldots,x_n)=f(\sigma x_1,\ldots,\sigma x_n),
	\] 
	where $\sigma x_j=\sum_{i=1}^n\sigma_{ij}x_i$. The $\GL_n$-invariant subspace of $\mathbb{F}_q [x_1, \ldots, x_n]$ was already known since the seminal work of Dickson \cite{Dickson11} in the early 20th century. It is again a polynomial algebra on $n$ generators:  
	\[
	\mathbb{F}_q [x_1, \ldots, x_n]^{\GL_n} = \mathbb{F}_q [Q_{n,0}, \ldots, Q_{n,n-1}], 
	\]
where the Dickson invariant $Q_{n,s}$, $0 \leq s \leq n-1$, is constructed in terms of determinants (the hat signifies that the corresponding entry is omitted):
\[
Q_{n,s}=\frac{\begin{vmatrix}
				x_1^{q^0} & \cdots &x_n^{q^0} \\
				\vdots & \ddots & \vdots\\
				\widehat{x_1^{q^{s}}} & \cdots & \widehat{x_n^{q^{s}}}\\
    \vdots & \ddots & \vdots\\
					x_1^{q^n}  & \cdots & x_n^{q^n}
		\end{vmatrix}}{
			\begin{vmatrix}
				x_1 & \cdots &x_n \\
				x_1^q & \cdots &x_n^q \\
				\vdots & \ddots & \vdots\\
				x_1^{q^{n-1}} & \cdots & x_n^{q^{n-1}}
		\end{vmatrix}}.
  \]
 This ring of invariants appears in many different contexts and has become part of the standard toolbox for algebraic topology and group cohomology, among others. The rings of invariants for parabolic subgroups of $\GL_n$ were also known, (see Hewitt \cite{Hewett96} and Kuhn-Mitchell \cite{Kuhn1986}), and they are also of interest in many different areas.    
	
Now fix a positive integer $m$ and let $I_{n,m}$ denote the ideal of $(x_1^{q^m},\ldots,x_n^{q^m})$ in $\F_q[x_1,\ldots,x_n]$. We denote by $\Qcal_m(n)$ the quotient ring $\F_q[x_1,\ldots,x_n]/(x_1^{q^m},\ldots,x_n^{q^m})$. Since $I_{n,m}$ is stable under the action of $\GL_n$, there is an induced action of $\GL_n$ on the quotient $\Qcal_m(n)$. The $\GL_n$-module structure of the quotient ring $\Qcal_m (n)$ is used in Kuhn's analysis \cite{Kuhn1987} of the rank of Morava $K$-theory of finite groups. When $q=2$, $\Qcal_m (n)$ can be realized as the mod 2 cohomology of a manifold which is the $n$-fold product of the truncated real projective space $\mathbb{R}P^{2^m-1}$. The action of $\GL_n$ and of the Steenrod algebra on this manifold has been under intense investigation in algebraic topology (see Boardman \cite{Boardman93}, Meyer and Smith \cite{Meyer-Smith-2005} or the 2-volume books by Walker and Wood \cite{Walker-Wood-V1, Walker-Wood-V2} and the references therein). Despite some serious attempts, the structure of the invariant ring of $\Qcal_m (n)$ remains largely mysterious until the work of Lewis, Reiner and Stanton \cite{LewisReinerStanton2017} from a completely new perspective. 
	
	We need to introduce some definitions and notations. Recall a weak composition $\beta=(\beta_1,\ldots,\beta_{\ell})$ is just a sequence of non-negative integers. Denote by $B_i$ the partial sum $B_i = \sum_{j=1}^{i} \beta_j$ and write $|\beta|$ for the total $\sum_{1}^{\ell} \beta_i$. Define a partial order on the set of weak compositions by declaring that $\beta \leq \beta'$ iff $\beta_i \leq \beta_i'$ for all $i$. A useful way to organize information about a graded vector space $M$ over a field $k$ is through its Hilbert (or Hilbert-Poincar\'e) series $H(M;t) = \sum_{i=0}^{\infty} (\dim_k M_i) t^i$. Very often, it is possible to read off other helpful information about $M$ from its Hilbert series, such as the degree of its generators and their distribution patterns. For graded objects over the finite field $\mathbb{F}_q$, it is natural to use the following generalization of the usual binomial and $q$-binomial coefficients: If $\alpha = (\alpha_1, \ldots, \alpha_{\ell})$ is a composition of a positive integer $k$, let $A_i$ denote the partial sum $A_i = \alpha_1 + \ldots + \alpha_i$. Following \cite{Reiner-Stanton2010}, the $(q,t)$-multinomial coefficient  $ \bmt{k\\ \alpha}_{q,t}$ is defined by the formula   
 \[
 \bmt{k\\ \alpha}_{q,t} = \frac{\prod_{j=0}^{k-1} (1-t^{q^k-q^{j}})} {\prod_{i=1}^{\ell} \prod_{j=1}^{\alpha_i} (1-t^{q^{A_i} - q^{A_{i} -j}}) }.
 \]
This definition can be easily extended to weak compositions in an obvious way. Lewis, Reiner, and Stanton proposed the following: 
	\begin{conj}[Parabolic conjecture \cite{LewisReinerStanton2017}] 
		Let $n$ be a positive integer, $\alpha$ be a composition of $n$ and let $\Par_{\alpha}$ denote the corresponding parabolic subgroup of $\GL_n (\mathbb{F}_q)$. The Hilbert series for the graded $\mathbb{F}_q$-vector space of $\Par_{\alpha}$-invariants, $\Qcal_m (n)^{\Par_{\alpha}}$, is the power series $C_{\alpha,m}(t)$ given by 
		\[
		C_{\alpha,m}(t)=\sum_{\beta\le \alpha, |\beta|\le m }t^{e(m,\alpha,\beta)}\bmt{m\\ \beta, m-|\beta| }_{q,t},
		\] 
		where  $e(m,\alpha,\beta) = \sum_{i=1}^{n} (\alpha_i-\beta_i)(q^m-q^{B_i})$ and $B_i=\beta_1+\cdots +\beta_i$.
  \end{conj} 
 
 For example, when $m \geq n=2$ and $\alpha = 1^2$  so that $\Par_{\alpha}$ is the Borel subgroup of $\GL_2 (\mathbb{F}_q)$, then $\beta \in \{(0,0), (0,1), (1,0), (1,1) \}$, and the conjectured Hilbert series for $\Qcal_2(m)^{\B_2}$ equals:  
 \[
t^{2(q^m-1)} + t^{q^m-1} \frac{1-t^{q^m-1}}{1-t^{q-1}} + t^{q^m-q}  \frac{1-t^{q^m-1}}{1-t^{q-1}} + \frac{(1-t^{q^m-1})(1-t^{q^m-q})}{(1-t^{q-1})(1-t^{q^2-q})}.
 \]
Lewis, Reiner and Stanton provided evidence for their conjectures by verifying the case $n=2$ for all $m$.  
In \cite{Goyal18}, Goyal obtained some information about the parabolic conjecture for small $m$ and $n$. For example, he constructed some explicit family of exotic invariants that do not come from the usual polynomial invariants. A version of the parabolic conjecture was investigated in \cite{Drescher-Shepler-2020} by Drescher and Shepler. Recently, in the work of Deng \cite{Deng23}, invariants and coinvariants of the truncated polynomial ring have found applications in the study of torsion classes in the cohomology of $\SL_2 (\mathbb{Z})$.

The aim of this paper is to describe explicitly a linear basis for the space of invariants of $\Qcal_m(n)$ under the action of the smallest parabolic subgroup - the Borel subgroup $\B_n$ corresponding to the composition $\alpha = 1^n$. As a corollary, one can easily compute its Hilbert series, and verify the parabolic conjecture for the Borel subgroup for all $m,n \geq 1$.

	To state the main results, we need the following:
	
	\begin{definition}Let $a,b,c$ be positive integers such that $1\le a\le c+1$. Define an operator
		$$\delta_{a;b} \colon \F_q(x_1,\ldots,x_{c})\to \F_q(x_1,\ldots,x_{c+1})$$ 
	 from the ring of rational functions in $c$ variables to the ring of rational functions in $c+1$ variables as follows:  
	If $f\in \F_q(x_1,\ldots,x_{c})$, then $\delta_{a;b}(f)$ is defined as the quotient 
		$$\delta_{a;b}(f)=\frac{\begin{vmatrix}
				x_1 & \cdots &x_a \\
				x_1^q & \cdots &x_a^q \\
				\vdots & \ddots & \vdots\\
				x_1^{q^{a-2}} & \cdots & x_a^{q^{a-2}}\\
					x_1^{q^b} f(\widehat{x_1},x_2,\ldots,x_{c+1})& \cdots & x_a^{q^b}f(x_1,\ldots,\widehat{x_a},\ldots,x_{c+1})
		\end{vmatrix}}{
			\begin{vmatrix}
				x_1 & \cdots &x_a \\
				x_1^q & \cdots &x_a^q \\
				\vdots & \ddots & \vdots\\
				x_1^{q^{a-1}} & \cdots & x_a^{q^{a-1}}
		\end{vmatrix}}.$$
\end{definition}

Thus $\delta_{a;b}$ increases degree by $q^b-q^{a-1}$ if $b\ge a-1$. In particular, when $a=1$, the formula is simplified to
$$\delta_{1;b}(f)=x_1^{q^b-1}\cdot f(x_2,\ldots,x_{c+1}).$$	
	
When $f=1$ is the constant function and $b\ge a-1$, $\delta_{a;b} (1)$ is $S_\lambda(x_1,\ldots,x_a)$, the Macdonald's 7th variation of Schur functions (\cite{Macdonald-Schur}) corresponding to the partition $\lambda=(b-a+1)$. In particular, $\delta_{a;a} (1)$ is the usual Dickson invariant $Q_{a,a-1}$ \cite{Dickson11} of degree $q^a-q^{a-1}$. For ease of notation, we write $D_a$ for $\delta_{a;a}(1)$.

	\begin{definition} \label{Y-inv}
	For two sequences $I=(i_1,\ldots,i_k)$ and $J=(j_1,\ldots,j_k)$ of non-negative integers, define the rational function
	$$Y_b(I; J) :=  \delta_{1;b}^{i_1}( D_1^{j_1} \delta_{2;b}^{i_2} (  D_2^{j_2}  (   \cdots    ( \delta_{k;b}^{i_k}(D_k^{j_k})   )  \cdots  )  ) ).
	$$
	We also define the Frobenius-like operator $\Phi$ by 
	$\Phi Y_b(I;J)=Y_{b+1}(0,I;0,J),$ or explicitly,
	$$\Phi Y_b(I;J)= \delta_{2;b+1}^{i_1}( D_2^{j_1} \delta_{3;b+1}^{i_2} (  D_3^{j_2}  (   \cdots    ( \delta_{k+1;b+1}^{i_k}(D_{k+1}^{j_k})   )  \cdots  )  ) ).$$ 
	\end{definition}

$\delta_{a;b} (f)$ is certainly not always a polynomial function, but we will see later that $Y_b(I; J)$ is indeed a polynomial for all $I$ and $J$.

For each integer $a \geq 0$, we denote by $[a]_q$ the $q$-integer 
$$[a]_q=\frac{q^a-1}{q-1}.$$

\begin{definition}\label{B-basis} For $n\ge 1, m\ge 0$, define the set $\Bcal_{m}(n)$ by the following inductive rule:
	\begin{eqnarray*}
	\Bcal_0(n)&=&\{1\}, \;\text{for all} \; n\ge 1; \\
	\Bcal_m(1)&=&\{ D_1^a\mid a\le [m]_q \}, \;\text{for all} \; m\ge 0;\\
		\Bcal_{m}(n)&=&\{\delta_{1;m}(Y) \mid Y\in \Bcal_m(n-1) \} \sqcup
	\{ D_1^a \Phi(Y) \mid  a<[m]_q, Y\in \Bcal_{m-1}(n-1) \}, n\ge 2, m\ge 1.
	\end{eqnarray*}
\end{definition}
More concretely, we will show that 
\begin{prop}\label{B(m,n) description} 
    $\Bcal_{m}(n)$ is the disjoint union $\coprod_{k=1}^{\min(n,m+1)}\Bcal_m^k(n)$, where $\Bcal_m^k(n)$ denotes the set consisting of all elements $Y_m(I; J)$ for which the sequences $I=(i_1,\ldots,i_k)$ and $J=(j_1,\ldots,j_k)$ satisfy the following conditions:
$$\begin{cases}
	i_1+\cdots+i_k=n-k,\\ 
	j_1 < [m]_q, 
	\ldots,
	j_{k-1}< [m-k+2]_q,
	j_k\le [m-k+1]_q.
\end{cases}$$
\end{prop}
 
We are now ready to state the main result of this paper:
	
	\begin{thm}\label{main} $\Bcal_{m}(n)$ is a
	basis for the $\mathbb{F}_q$-vector space of $\B_n$-invariants $\Qcal_m(n)^{\B_n}$.
	\end{thm}
Once a basis is established, the Hilbert series can be computed and we have:   	
\begin{cor}
    The Lewis-Reiner-Stanton parabolic conjecture \cite[1.5]{LewisReinerStanton2017} about the Hilbert series of the invariant subspace of $\Qcal_m (n)$ for the Borel subgroup is true for all $m \geq 0$ and $n \geq 1$.   
\end{cor}

Recall that the truncated algebra $\Qcal_m (n)$ admits a $\GL_n$-invariant, non-degenerate $\mathbb{F}_q$-bilinear pairing: 
\[
(x_1^{i_1} \ldots x_n^{i_n}, x_1^{j_1} \ldots x_n^{j_n}) = 
\begin{cases}
    x_1^{q^m-1} \ldots x_n^{q^m-1} & \text{if $i_s + j_s = q^m-1$ for all $1 \leq s \leq n$,}\\
    0 & \text{otherwise}.
\end{cases}
\]
The duality resulting from this pairing yields information about the cofixed space $\Qcal_m (n)_{\B_n}$, and as observed in \cite[Corollary 3.5]{LewisReinerStanton2017}, by taking $m \to \infty$, implies 
\begin{cor}
    The Lewis-Reiner-Stanton parabolic conjecture \cite[1.6]{LewisReinerStanton2017} about the Hilbert series of the cofixed space $\F_q[x_1,\ldots,x_n]_{\B_n}$ of the Borel subgroup $\B_n$ is true for all $n \geq 1$. 
\end{cor}

	The following examples make the theorem explicit in some small ranks:

	\begin{example}
		For $m\ge 0$,  $\Qcal_m(1)^{\B_1}$ can be identified with the family 
		\begin{enumerate}
			\item $D_1^{a}$, \quad $a\le [m]_q$.
		\end{enumerate}
	\end{example}
	
	\begin{example}
		For $m\ge 1$, the following families form a basis for $\Qcal_m(2)^{\B_2}$:
		\begin{enumerate}
			\item $\delta_{1;m}(D_1^a)$, \quad $a\le [m]_q$;
			\item $D_1^{a}D_2^b$, \quad $a<[m]_q$, $b\le [m-1]_q$.
		\end{enumerate} 
  Compared with the Hilbert series for the case $n=2$ given in the introduction, we see that the first two summands correspond to the first family, and the last two summands correspond to the second one.
	\end{example}
	
	\begin{example} For $m\ge 2$, the following families form a basis for $\Qcal_m(3)^{\B_3}$:
		\begin{enumerate}
			\item $\delta_{1;m}\delta_{1;m}(D_1^a)$, \quad $a\le [m]_q$,
			\item $\delta_{1;m}(D_1^{a}D_2^b)$, \quad $a<[m]_q$, $b\le [m-1]_q$,
			\item $D_1^{a}\delta_{2;m}(D_2^b)$, \quad $a<[m]_q$, $b\le [m-1]_q$,
			\item $D_1^{a}D_2^{b}D_3^{c}$, \quad $a<[m]_q$, $b<  [m-1]_q$, $c\le [m-2]_q$.
		\end{enumerate}
When $m=1$, only the first three families are involved. 
	\end{example}
	
 The plan of the paper is as follows. In \S \ref{delta-op}, we introduce the operator $\delta_s$ which is essentially a construction of new invariant from old. We give formulas for $\delta_{a;b}$ and its iterates. These will be used in \S \ref{Y-Polynomiality} to prove the polynomiality of the rational functions $Y$ given in the main theorem. In \S \ref{Y-invariance}, we prove the invariance of $Y$'s under the action of the Borel subgroup. The main theorem is proved in \S \ref{main-proof} and we verify the Lewis-Reiner-Stanton conjecture for the Borel group in \S \ref{Hilbert}. 
 A bijection between our basis $\Bcal_m(n)$ and the set of $\B_n$-orbits of $\mathbb{F}_{q^m}^n$ is constructed in \S \ref{bijection}. 
 In the last section, we propose a refinement of the parabolic conjecture. More precisely, we construct an explicit set of invariants in terms of the operator $\delta$, the classical Dickson invariants, and Macdonald's 7th variation of Schur functions \cite{Macdonald-Schur}. These parabolic invariants have required Hilbert series, and we conjecture that they form an $\mathbb{F}_q$-basis for the invariant subspace under consideration.  

  \section*{Acknowledgement} Part of this work was done while the second author visited the Vietnam Institute for Advanced Study in Mathematics (VIASM). He wishes to thank the VIASM for its hospitality. The third author received funding from Hung Vuong University's Fundamental Research Program under project grant No. 03/2024. All three authors are very grateful to the anonymous referee for his comments and suggestions, which helped improve the manuscript's quality.

	\section{The operator $\delta$} \label{delta-op}
In this section, we investigate the main properties of the operator $\delta_{a;b}$. We begin with some standard constructions in modular invariant theory.    
For each $1\le k\le n$, following  \cite{Dickson11} and \cite{Mui75}, put $$L_k=L(x_1,\ldots,x_k)=\det(x_j^{q^{i-1}})_{1\le i,j\le k}$$ 
and $$V_k=V(x_1,\ldots,x_k)= L(x_1,\ldots,x_k)/L(x_1,\ldots,x_{k-1}) = \prod_{\lambda_1,\ldots,\lambda_{k-1}\in \F_q}(\lambda_1 x_1+\cdots +\lambda_{k-1}x_{k-1}+x_k).$$
 Recall the fundamental equation $$V(x_1,\ldots,x_k,X)= X^{q^k} + \sum_{i=0}^{k-1}(-1)^{k-i} Q_{k,i}(x_1,\ldots,x_k)X^{q^{i}}.$$
By construction, the polynomials $Q_{k,0},\ldots,Q_{k,k-1}$ are all invariant under the action of the general linear group $\GL_k$. A celebrated theorem due to Dickson \cite{Dickson11} stated that the invariant ring $\F_q[x_1,\ldots,x_k]^{\GL_k}$ is again a polynomial algebra, which is exactly $\F_q[Q_{k,0},\ldots,Q_{k,k-1}]$. Note that $V(x_1,\ldots,x_k,X)$ is linear in $X$.  
	
If $I=\{i_1<\cdots <i_k\}$ is a subset\footnote{We will always consider ordered subsets of indices in this paper.} of $[n]:=\{1,\ldots,n\}$, we write $L(I)$ for $L(x_{i_1},\ldots,x_{i_k})$.
	For two subsets $I,J$ of $[n]$, we put
	 $$V(J,I) =\prod_{i\in I} V(J,i),$$ 
	where $V(J,i)$ stands for $V(x_{j_1},\ldots,x_{j_k},x_i )$ if $J=\{j_1,\ldots,j_k \}.$ By convention, $V(J,\emptyset)=1$. Note that the order of elements in $I$ and $J$ is irrelevant in the definition of $V(J,I)$ and $V(J, I) = 0$ if $I$ and $J$ are not disjoint.

We will write $\delta_{r+1}$ for $\delta_{r+1;b}$ in the remaining part of this section.  
	Expanding along the last row of the determinant in the numerator of $\delta_{r+1}(f(x_1,\ldots,x_k))$, $k\ge r$, yields another description of $\delta_{r+1}$:
	\begin{eqnarray*}
		\delta_{r+1}(f)&=&\sum_{j=1}^{r+1} \frac{(-1)^{r+1+j}x_j^{q^b} f(x_1,\ldots,\widehat{x_j},\ldots,x_{k+1})L(x_1,\ldots,\widehat{x_j},\ldots,x_{r+1})}{L(x_1,\ldots,x_{r+1})}\\
		&=& \sum_{j=1}^{r+1} \frac{f(x_1,\ldots,\widehat{x_j},\ldots,x_{k+1})x_j^{q^b} }{V(x_1,\ldots,\widehat{x_j},\ldots,x_{r+1},x_j )}.
	\end{eqnarray*}
	
The following result provides a formula for the action of an iterated composition of $\delta_{r+1}$:

	\begin{prop} \label{iterate-delta}
		If $f$ is a rational function in $r' \ge r$ variables and $h$ is a positive integer,  then we have the following formula: 
		$$\delta_{r+1}^h(f)(x_1,\ldots,x_{r'+h}) =\sum_{\substack{I\subset [r+h] \\ |I|=h  }} \frac{f({\bar I }) \varphi^b (I) }{V(\bar I,I)},$$
		where $\bar I$ denotes the complement of $I$ in $[r+h]$, $\varphi$ is the Frobenius map, so that if $I = \{i_1 < \ldots <i_h\}$, then $\varphi^b (I) = (x_{i_1} \ldots x_{i_h})^{q^b}$, and 
  \[
  f(\bar I) = f(x_{1},\ldots,\widehat{x_{i_1}}, \ldots, \widehat{x_{i_h}}, \ldots,x_{r+h}, x_{r+1+h},\ldots,x_{r'+h}).
  \]
	\end{prop}

\begin{proof} We proceed by induction on $h$. The case $h=1$ is already discussed above.  For the inductive step, we have 
	\begin{eqnarray*}
	\delta_{r+1}(\delta_{r+1}^h(f)) &=&\sum_{j=1}^{r+1} (-1)^{r+1+j} \frac{L(x_1,\ldots,\widehat{x_j},\ldots,x_{r+1})  x_j^{q^b} }{L_{r+1}} \cdot 
\delta_{r+1}^h(f)(x_1,\ldots,\widehat{x_j},\ldots,x_{r'+h+1})\\
	&=&
	\sum_{j=1}^{r+1} (-1)^{r+1+j} \frac{L(x_1,\ldots,\widehat{x_j},\ldots,x_{r+1})  x_j^{q^b} }{L_{r+1}} 
	\sum_{\substack{I\subset [r+h+1]\setminus \{j\} \\ |I|=h} } \frac{f({\bar I} )  \varphi^b (I) }{V({\bar I},I)},
\end{eqnarray*}
$\bar I$ denoting the complement of $I$ in $[r+h+1]\setminus \{j\}$. 
Observe that each subset $J$ of $[r+h+1]$ of cardinality $h+1$ is of the form $J=I\sqcup \{j\}$ for some $1\le j\le r+1$, $I\subset [r+h+1]\setminus \{j\}$ and $|I|=h$.  We may then rewrite the above identity as follows:
\begin{eqnarray*}
	\delta_{r+1}^{h+1}(f) &=& \sum_{\substack{J\subset [r+h+1]\\ |J|=h+1}} \sum_{j\in J\cap [r+1]} (-1)^{r+1+j} \frac{L(x_1,\ldots,\widehat{x_j},\ldots,x_{r+1})}{L_{r+1}} 
	\frac{f({\bar J} )  \cdot \varphi^b (J)   }{V({\bar J}, J\setminus \{j\} )}\\
	&=&
	\sum_{\substack{J\subset [r+h+1]\\ |J|=h+1}} \sum_{j\in J\cap [r+1]} (-1)^{r+1+j} \frac{L(x_1,\ldots,\widehat{x_j},\ldots,x_{r+1}) V({\bar J}, j)   }{L_{r+1}} 
	\frac{f({\bar J} ) \cdot \varphi^b (J)   }{V({\bar J}, J )},
\end{eqnarray*}
$\bar J$ denoting the complement of $J$ in $[r+h+1]$. 
It suffices now to prove the identity 
$$\sum_{j\in J\cap [r+1]} (-1)^{r+1+j}L(x_1,\ldots,\widehat{x_j},\ldots,x_{r+1}) V({\bar J}, j)    =L_{r+1}.$$
We extend the sum on the left hand side to all $j \in [r+1]$. This does not change the sum because if $j\in [r+1]\setminus J$, then $j\in \bar J$ and $V({\bar J},j)=0$. But now the identity
$$\sum_{j=1}^{r+1} (-1)^{r+1+j}L(x_1,\ldots,\widehat{x_j},\ldots,x_{r+1}) V({\bar J}, j)    =L_{r+1}$$
follows easily from the fundamental equation $$V({\bar J},j)= x_j^{q^r} + \sum_{i=0}^{r-1}(-1)^{r-i} x_j^{q^{i}}Q_{r,i}({\bar J})$$
and the Laplace expansions $$\sum_{j=1}^{r+1} (-1)^{r+1+j}L(x_1,\ldots,\widehat{x_j},\ldots,x_{r+1})  x_j^{q^i}    = \begin{cases}
	L_{r+1} & i=r,\\
	0 & i<r.
\end{cases}$$
The proof is complete.
\end{proof}

We must also understand the behavior of iterated applications of $\delta_{r+1}$ for different $r$. Let $y_1, \ldots, y_h$ be $h$ indeterminates. Using the fundamental equation, expand the product  $\prod_{j=1}^h V(x_1,\ldots,x_s,y_j)$ and write the result in the form 
\[
\prod_{j=1}^h V(x_1,\ldots,x_s,y_j)  =  \sum_{T\in \Tcal(s,h)}\beta_{T}(x_1,\ldots,x_s) \cdot  \alpha_{T}(y_{1},\ldots,y_{h}),
\]
where $T$ runs over the finite set $\Tcal(s,h)$ of all weak compostions $h=t_0 + \ldots + t_{s}$ of $h$ into $s+1$ parts. Thus 
\[
\beta_T (x_1, \ldots x_s) = (-1)^{\sum_{i=1}^s t_i (s-i)} Q_{s,0}^{t_0} \ldots Q_{s,s-1}^{t_{s-1}},
\]
and $\alpha_T (y_1, \ldots, y_h)$ is the monomial symmetric function of $y_1, \ldots, y_h$ corresponding to the exponent set consisting of $q$-powers such that for each $0 \leq i \leq s$, $q^i$ appears exactly $t_i$ times. Thus by definition, $\alpha_T (y_1, \ldots, y_h)$ is multilinear, symmetric and divisible by the product $y_1 \ldots y_h$. 

When $h=0$, the set $\Tcal (s,0)$ has only one element where $t_i = 0$ for all $i$. We extend the definition of $\beta_T$ and $\alpha_T$ to this case by setting $\alpha_T = \beta_T=1$.

\begin{definition}\label{diagonal}
For each rational function $g$ in $r' \ge r$ variables and each $T \in \Tcal (s,h)$, define
\begin{eqnarray*}
	A_{r;T}(g)&=&\sum_{\substack{I\subset [r+h]\\ |I|=h} }\frac{g(\bar I)\alpha_{T}(I)}{V(\bar I,I) },
\end{eqnarray*}
where $\alpha_T(I)=\alpha_T(x_{i_1},\ldots,x_{i_h})$ if $I=\{i_1<\cdots <i_h \}$, $\bar I$ denotes the complement of $I$ in $[r+h]$, and 
\[
g(\bar I)=g(x_{1},\ldots,\widehat{x_{i_1}}, \ldots, \widehat{x_{i_h}}, \ldots,x_{r+h}, x_{r+1+h},\ldots,x_{r'+h}).
\]
 
\end{definition}
 The operator $A_{r;T}$ generalizes $\delta$. Indeed, when $\tau  \in \Tcal (b,h)$ is the composition where $\tau_b = h$ and all others $\tau_i$ equal zero, then by Proposition \ref{iterate-delta}, we get $A_{r;\tau} (g) = \delta_{r+1}^h (g)$. It will be shown later that very often, $A_{r;T} (g)$ is a polynomial function.  
	
 The following will be useful in inductive arguments. 
	
	\begin{prop}\label{composite-delta}
	Let $r,s,k$ be positive integers such that $r\le s+k$. Let $f(x_1,\ldots,x_{s'})$ and $g(x_1,\ldots,x_{r'})$ be rational functions where $s'\ge s$, $r'\ge r$. The following equality holds:  
	$$\delta_{r+1}^h(g\cdot \delta_{s+1}^k(f))= \sum_{T \in \Tcal(s,h)} A_{r;T}(g) \cdot 	\delta_{s+1}^{h+k}(\beta_{T} f).$$
\end{prop}

	\begin{proof} We consider the case $s'=s$. When $s'>s$, one just shifts the variables of $f$ in an appropriate way.
		We have
		\begin{eqnarray*}
			\delta_{r+1}^h(g\cdot \delta_{s+1}^k(f)) 
			&=&
			\sum_{\substack{I\subset [r+h]\\ |I|=h} }\frac{ g({\bar  I})  \varphi^b (I) }{V(\bar I,I)} \cdot  \delta_{s+1}^k(f)({\bar I},{r+h+1},\ldots,{s+k+h} ),  \\
			&=&
			\sum_{\substack{I\subset [r+h]\\ |I|=h} }\frac{ g({\bar I})  \varphi^b (I) }{V(\bar I,I)} \cdot  \sum_{\substack{J\subset I' \\ |J|=k  }} \frac{f({I'\setminus J })  \varphi^b (J) }{V(I'\setminus J,J)  }, \quad \text{(putting $I':=[s+k+h]\setminus I$)}\\
			&=&
			\sum_{\substack{H\subset [s+k+h]\\ |H|=h+k} }
			\sum_{\substack{I\subset H \cap [r+h] \\ |I|=h} }\frac{g({\bar I}) }{V(\bar I,I) }
			\frac{f({\bar H })  \varphi^b (H) }{V(\bar H,H)  } V(\bar H,I),   \quad (\text{putting $H=I\sqcup J$, need $r+h\le s+k+h$ })\\
			&=&
			\sum_{\substack{H\subset [s+k+h]\\ |H|=h+k} }
			\sum_{\substack{I\subset [r+h]  \\ |I|=h} }\frac{g({\bar I}) }{V(\bar I,I) }
			\frac{f({\bar H })  \varphi^b (H) }{V(\bar H,H)  }V(\bar H,I),\\
			&=&\sum_{T\in \Tcal(s,h)}
			\sum_{\substack{H\subset [s+k+h]\\ |H|=h+k} }
			\sum_{\substack{I\subset [r+h]  \\ |I|=h} }\frac{g({\bar I})  \alpha_T(I)}{V(\bar I,I) }
			\frac{\beta_T(\bar H)f({\bar H })  \varphi^b (H) }{V(\bar H,H)  } \quad (\text{expanding $V(\bar H,I)$}).
		\end{eqnarray*}
In the 4th equality, the condition $I \subset H\cap [r+h]$ is relaxed to $I \subset [r+h]$. This is possible since if $I \not\subset H$, then $\bar H\cap I\not =\emptyset$ and   $V(\bar H,I)=0$. The result now follows by using Definition \ref{diagonal}. 
\end{proof}

	\section{Polynomiality of $Y$} \label{Y-Polynomiality}
In this section, we prove that the family of rational functions $Y_m (I;J)$ are in fact polynomial. 
 	
Given two symmetric rational functions $f(x_1,\ldots,x_r)$ and $g(x_1,\ldots,x_h)$, the \textit{weighted shuffle product} of $f$ and $g$ is a function in $r+h$ variables, defined by the formula 
	$$f\bullet g:= \sum_{\substack{I\sqcup J= [r+h] \\ |I|=r, |J|=h   }} \frac{f(I)g(J) }{V( I, J)}.$$

Note that this shuffle product depends on $r$ and $h$. The weighted shuffle product generalizes further our previous construction. For example, if $f$ is a rational function in $r$ variables, then Proposition \ref{iterate-delta} says that $$\delta_{r+1;b}^h(f)=f\bullet \varphi^b_h,$$
where $\varphi_h$ is the Frobenius map on $h$ variables. Similarly, if $g$ is a rational function in $r$ variables, then the rational function $A_{r;T}(g)$ defined in Definition \ref{diagonal} can be written as
$$A_{r;T}(g)=g\bullet \alpha_T.$$

It follows from Proposition \ref{composite-delta} that $Y$ is a sum of a product of weighted shuffle products. The following provides a criterion for the polynomiality of a 
weighted shuffle product.

	\begin{prop} \label{delta-poly}
		If $f(x_1,\ldots,x_s)$ is a $\GL_s$-invariant polynomial and $g(x_1,\ldots,x_{k-s})$ is a polynomial which is symmetric, multilinear and divisible by $x_1\cdots x_{k-s}$, then $f\bullet g$ 
		is a polynomial in $\F_q[x_1,\ldots,x_k]$.
	\end{prop}

	\begin{proof} 
		For each $I \subset [k]$, $V(I, J)$ factors completely into linear forms of the form 
		\[
		x_j + \sum_{i \in I} \lambda_{i} x_{i}, \quad j \in J, \lambda_i \in \mathbb{F}_q.  
		\]
After taking the common denominator and reducing the quotient, we want to show that any such linear form does not appear in the denominator of $f \bullet g$. By symmetry and without loss of generality, it suffices to consider the linear form  
		\[
		\omega = x_{a+1} + \lambda_1 x_1 + \ldots + \lambda_a x_a, 
		\]
		where $a \leq s$, and $\lambda_i \neq 0$ for all $1 \leq i \leq a$. If $V(I, \bar I)$ contains $\omega$ as a factor (up to a non-zero scalar), then $I$ must have the form $I = [a+1] \slash \{i\} \sqcup K$ for some subset $K \subset [k]$ such that $|K|=s-a$ and $[a+1] \cap K = \emptyset$. 
		
		Thus we only need to consider the summands of $f \bullet g$ corresponding to these subsets: 
		\begin{equation}\label{sum over K}
		\sum_{|K|=s-a, [a+1] \cap K = \emptyset}  \sum_{i=1}^{a+1} \frac{f([a+1] \slash \{i\} \cup K) g(\bar K\slash  [a+1] \cup \{i\})}{V([a+1] \slash \{i\} \cup K, i) V([a+1] \slash \{i\} \cup K, \bar K\slash  [a+1] )}.
		\end{equation}
		Note that 
		\[
		\frac{1}{V([a+1] \slash \{i\} \cup K, i)} = \frac{L([a+1]\slash \{i\} \cup K)}{L([a+1] \slash \{i\} \cup K \cup  i)}.
		\]
		After rearranging the columns in the usual order, noting that every element of $K$ is greater than $a+1$, we have 
		\[
		L([a+1] \slash \{i\} \cup K \cup  i) =
		\begin{cases}
			(-1)^{s-i} L([a+1]\cup K)   & i \leq a,\\
			(-1)^{s-a} L([a+1]\cup K) & i =a+1.
		\end{cases}
		\]
		Hence the summand corresponding to each $K$ in equation \eqref{sum over K} can be rewritten as 
		\begin{multline*} 
		\frac{1}{L([a+1]\cup K)} \Big[ \sum_{i=1}^{a} \frac{(-1)^{s-i} f([a+1] \slash \{i\} \cup K) g(\bar K\slash  [a+1] \cup \{i\}) L([a+1]\slash \{i\} \cup K)}{V([a+1] \slash \{i\} \cup K, \bar K\slash  [a+1] )}+ \\
			\frac{(-1)^{s-a} f([a] \cup K) g(\bar K\slash  [a+1] \cup \{a+1\}) L([a] \cup K)}{V([a]\cup K, \bar K\slash  [a+1])} \Big].
		\end{multline*}
		Observe that $\omega$ appears in the factorization of $L([a+1] \cup K)$ with  multiplicity $1$, and the other factors $V([a+1] \slash \{i\} \cup K, \bar K\slash  [a+1] )$ clearly are not divisible by $\omega$ for $1 \leq i \leq a+1$. Thus it suffices to show that for each $K$ fixed, the sum 
		\begin{multline}\label{sum equal zero}
		\sum_{i=1}^{a} \frac{(-1)^{s-i} f([a+1] \slash \{i\} \cup K) g(\bar K\slash  [a+1] \cup \{i\}) L([a+1]\slash \{i\} \cup K)}{V([a+1] \slash \{i\} \cup K, \bar K\slash  [a+1] )}+ \\
			\frac{(-1)^{s-a} f([a] \cup K) g(\bar K\slash  [a+1] \cup \{a+1\}) L([a] \cup K)}{V([a]\cup K, \bar K\slash  [a+1])} 
		\end{multline}
		vanishes when $\omega = 0$. Indeed, since each $x_i$ appears non-trivially in $\omega$ for all $1 \leq i \leq a$, $x_{a+1}$ can be written as a nontrivial linear combination of $x_1 \ldots x_a$, and since $f$ is a $\GL_s$-invariant, we have  
		\[
		f([a+1]\slash \{i\},K) = f([a] \cup K), \quad \text{for all $1 \leq i \leq a \leq s$.}
		\] 
	 We also have    
		\[
		L([a+1]\slash \{i\} \cup K)=
		\begin{cases}
			(-1)^{a-i} \lambda_i L([a]\cup K)   & i \leq a,\\
			L([a]\cup K) & i =a+1,
		\end{cases}
		\]
		and
		\[
		V([a+1]\slash \{i\} \cup K,\bar K\slash [a+1]) = V([a] \cup K,\bar K\slash [a+1]).
		\]
 Thus the sum in	Equation \eqref{sum equal zero}, up to a sign, equals 
		\begin{equation*}
		\frac{f([a] \cup K) L([a] \cup K)}{V([a] \cup K,\bar K\slash [a+1])} \Big[\sum_{i=1}^{a} \lambda_i g(\bar K\slash [a+1] \cup \{i\}) + g(\bar K\slash [a+1] \cup \{a+1\}) \Big].
		\end{equation*}
Finally, because $g$ is both symmetric and multilinear, we have 
\[
\sum_{i=1}^{a} \lambda_i g(\bar K\slash [a+1] \cup \{i\}) + g(\bar K\slash [a+1] \cup \{a+1\}) = g(\bar K\slash [a+1] \cup \{\omega\}) =0. 
\]
The proof is finished. 
	\end{proof}

\begin{cor} For any pair of $k$-tuples $I$ and $J$ of non-negative integers,  $Y_b(i_1,\ldots,i_k; j_1,\ldots,j_k)$ given in Definition \ref{Y-inv} is a polynomial.
\end{cor}

\begin{proof} Recall that $$Y_b(I; J) =  \delta_{1;b}^{i_1}( D_1^{j_1} \delta_{2;b}^{i_2} (  D_2^{j_2}  (   \cdots    ( \delta_{k;b}^{i_k}(D_k^{j_k})   )  \cdots  )  ) ).
	$$
	For simplicity, write $\delta_r$ for $\delta_{r;b}$. We will prove a more general result that this expression is still a polynomial even after replacing each $D_a^{j_a}$  by a polynomial $u_a$ which is $\GL_{a-1}$-invariant in the first $(a-1)$ variables.  
	 Consider the case where $Y=\delta_k^{i_k}(u_k)$. If $i_k=0$ then we have nothing to prove. Suppose now $i_k>0$. Write $u_k$ in the form $\sum_r f_r(x_1,\ldots,x_{k-1})g_r(x_k,x_{k+1,}\ldots )$ where $f_r(x_1,\ldots,x_{k-1})$ are  $\GL_{k-1}$-invariant. Then  
	$$\delta_k^{i_k}(u_k)=\sum_r \delta_k^{i_k}(f_r)\cdot g_r(x_{k+i_k},x_{k+1+i_k},\ldots ).$$
	The polynomiality of $Y$ then follows from that of $ \delta_k^{i_k}(f_r)$, which is an application of Proposition \ref{delta-poly}.
	
	For the inductive step, using Proposition \ref{composite-delta}, we have 
	\begin{eqnarray*}
		\delta_{s}^{i_s}( u_s\delta_{s+1}^{i_{s+1}} (  u_{s+1}  (   \cdots    ( \delta_{k}^{i_k}(u_k)   )  \cdots  )  ) )&=& \sum_{T\in \Tcal(s,i_s)} A_{s-1;T}(u_s) 
		\delta_{s+1}^{i_s+i_{s+1}}( \beta_{T}u_{s+1}\delta_{s+2}^{i_{s+2}} (  u_{s+2}  (   \cdots    ( \delta_{k}^{i_k}(u_k)   )  \cdots  )  ) ).
	\end{eqnarray*}
 Since each $\beta_{T}u_{s+1}$ is invariant with respect to the first $s$ variables, we see that 
 \[
 \delta_{s+1}^{i_s+i_{s+1}}( \beta_{T}u_{s+1}\delta_{s+2}^{i_{s+2}} (  u_{s+2}  (   \cdots    ( \delta_{k}^{i_k}(u_k)   )  \cdots  )  ) )
 \] 
 is a polynomial by inductive hypothesis. By Proposition \ref{delta-poly}, $A_{s-1;T}(u_s) $ is also a polynomial, and so the result is proved.
\end{proof}
\section{Invariance of $Y$} \label{Y-invariance}
In this section, we prove that $Y_m(I;J)$, considered as an element in $\Qcal_{m} (n)$ is invariant under the action of the Borel subgroup. We need the following definition.
\begin{definition}
	Let $f(x_1,\ldots,x_k)\in \F_q[x_1,\ldots,x_k]$ be a polynomial. We say that $f$ is \textbf{$(k,m)$-invariant} if the following conditions hold:
	\begin{enumerate}
		\item $f(\lambda_1x_1,\ldots,\lambda_kx_k)=f(x_1,\ldots,x_k)$ for all $\lambda_i\in \F_q^\times $. In other words, $f$ is  invariant under the action of the torus subgroup;
		\item $f(x_1,\ldots,x_i,\ldots, x_j+x_i,\ldots, x_k)=f(x_1,\ldots,x_k) +(x_i^{q^m})$ for all 
		$1\le i<j\le k.$
	\end{enumerate}
\end{definition}

\begin{rmk}
It is clear that $f$ satisfies the first condition if and only if it is a sum of monomials in which all exponents are multiple of $(q-1)$. Also, if $f$ is $(k,m)$-invariant then $f$ is $\B_k$-invariant modulo $(x_1^{q^m},\ldots,x_k^{q^m})$.
\end{rmk}

	\begin{prop}\label{k-invariant}
Suppose $r\le k+1$,  $f(x_1,\ldots,x_k)$ is $(k,m)$-invariant and $\delta_{r;m}(f)$ is a polynomial. Then $\delta_{r;m}(f)$ is $(k+1,m)$-invariant.
	\end{prop}
	
	\begin{proof} 	Let $N$ denote the numerator of $\delta_{r;m} (f)$:  
 $$N  = \begin{vmatrix}
				x_1 & \cdots &x_r \\
				\vdots & \cdots & \vdots\\
				x_1^{q^{r-2}} & \cdots & x_r^{q^{r-2}}\\
				x_1^{q^m} f(\widehat{x_1},x_2,\ldots,x_{k+1})& \cdots & x_r^{q^m}f(x_1,\ldots,\widehat{x_r},\ldots,x_{k+1})
		\end{vmatrix}.$$
		Consider the operator $\sigma$ which sends $x_j$ to $x_j+x_i$ for some $1\le i<j\le k+1$ and leaves fixed all other $x_{\ell}$, $\ell \neq j$. Since $\delta_{r;m}(f)$ is a polynomial, so is $\sigma(\delta_{r;m}(f))$. But the denominator $L_r$ is invariant under the action of $\sigma$, it follows that $\sigma(N)-N$ is divisible by $L_r$.     
  
  We will prove that $\sigma(N)-N$ is also divisible by $x_i^{q^m}$ if $j>r$, and by $x_i^{q^m+1}$ if $j \leq r$. Thus, in both cases, $\sigma(N)-N$ is divisible by $x_i^{q^m} L_r$, and we are done. 
  
  Now assume that $j \leq r$, the $j$th column of $\sigma N$ is a sum of two column vectors:
\[
[x_j, x_j^{q}, \ldots, x_j^{q^{r-2}}, x_j^{q^m} f(x_1, \ldots, \widehat{x_j}, \ldots, x_{k+1})]^T + [x_i, x_i^{q}, \ldots, x_i^{q^{r-2}}, x_i^{q^m} f(x_1, \ldots, \widehat{x_j}, \ldots, x_{k+1})]^T.
\]
Hence, $\sigma N - N$ is the sum of two determinants. The first have the top $(r-1)$ rows exactly the same as that of $N$, while the last row equals 
		\[
		\begin{cases}
			x_\ell^{q^m}[f(x_1,\ldots,\widehat{x_\ell},\ldots,x_j + x_i, \ldots, x_{k+1}) - f(x_1,\ldots,\widehat{x_\ell},\ldots,x_j, \ldots, x_{k+1})]    & \text{in column $\ell \neq j$,}\\
			0   & \text{in column $\ell=j$ if $j \leq r$.}\\
		\end{cases}
		\]
The difference $[f(x_1,\ldots,\widehat{x_\ell},\ldots,x_j + x_i, \ldots, x_{k+1}) - f(x_1,\ldots,\widehat{x_\ell},\ldots,x_j, \ldots, x_{k+1})]$ is divisible by $x_i^{q^m}$ because $f$ is a $(k,m)$-invariant. Hence all entries in the last row of $\sigma N - N$ are divisible by $x_i^{q^m}$.
By inspecting the $i$-th column, we see that this determinant is divisible by $x_i^{q^m+1}$.

For the second determinant, its $i$th and $j$th columns are almost the same, and this allows us to compute this determinant explicitly:  
\[
x_i^{q^m}[f(x_1,\ldots, x_i, \ldots, \widehat{x_j},\ldots,x_{k+1})-f(x_1,\ldots, \widehat{x_i},\ldots,x_j+x_i, \ldots, x_{k+1})] L_{r-1} (x_1, \ldots, \widehat{x_j}, \ldots, x_r). 
\]
Since $L_{r-1} (x_1, \ldots, \widehat{x_j}, \ldots, x_r) $ contains $x_i$, we conclude that this determinant is also divisible by $x_i^{q^{m}+1}$.

		

  In the case where $j >r$, only the first determinant occurs, and the proof is similar. The proof is complete. 
	\end{proof}

		\begin{cor}Each polynomial $Y_m(I;J)$ given in the main theorem is a $\B_n$-invariant modulo $(x_1^{q^m},\ldots,x_n^{q^m})$.
	\end{cor}

\begin{proof}Recall that 
$Y_m(I; J) :=  \delta_{1;m}^{i_1}( D_1^{j_1} \delta_{2;m}^{i_2} (  D_2^{j_2}  (   \cdots    ( \delta_{k;m}^{i_k}(D_k^{j_k})   )  \cdots  )  ) )$. Since $D_k^{j_k}$ is $(k,m)$-invariant and $\delta_{k;m}^{i_k}(D_k^{j_k})$ is a polynomial, it follows that 
$\delta_{k;m}^{i_k}(D_k^{j_k})$ is $(k+i_k,m)$-invariant. By repeating this argument, we see that $Y_m(I; J)$ is $(k+i_1+\cdots+i_k,m)$-invariant. But since $k+i_1+\cdots+i_k=n$, we are done. (See Section \ref{Hilbert} for the condition imposed on the sequences $I,J$ for which $Y_m(I,J)$ is an element of the set $\Bcal_m(n)$.)
\end{proof}

We conclude this section with a warning that the operator $\delta_{r;m}$ is not well-behaved with respect to taking modulo $(x_1^{q^m}, \ldots, x_n^{q^m})$. For example when $r=2$, we have $x_1^{q^m+q-2} \equiv 0$ modulo  $(x_1^{q^m}, x_2^{q^m})$ but   
\[
\delta_{2;m} (x_1^{q^m+q-2}) = x_1^{q^m-1} x_2^{q^m-1} \not\equiv 0 \quad \text{in $\Qcal_m (2)$.}
\]
	\section{Proof of Theorem \ref{main}}\label{main-proof}
In this section, we give a proof of the main theorem that the set $\Bcal_m (n)$ of polynomials $Y_m (I, J)$ form a basis for the $\mathbb{F}_q$-vector space $\Qcal_m (n)^{\B_n}$. For inductive arguments, we make some observations about monomials with smallest $x_1$-degree occurring in a $\B_n$-invariant polynomial in $\Qcal_m (n)$.  

Since the case $n=1$ is trivial, we assume in this section that $n \geq 2$. We write $\B(x_2,\ldots,x_n)$ for the image of the inclusion $b\mapsto \pmt{1 & 0 \\ 0 &b}$ of the Borel subgroup $\B_{n-1}$ into $\GL_n$.

Let $F(x_1, \ldots x_n)$ be a $\B_n$-invariant polynomial in $\Qcal_m (n)$. Note that all exponents occurred in $F$ are multiple of $(q-1)$. If $F$ is divisible by $x_1^{q^m-1}$, then we have the following simple observation: 
\begin{lem}\label{top-type} Suppose $m\ge 0$, and $F(x_1,\ldots,x_n)=x_1^{q^m-1}f(x_2,\ldots,x_n)$ is a polynomial which is $\B_n$-invariant in $\Qcal_m(n)$. Then $f(x_2,\ldots,x_n)$ is $\B(x_2,\ldots,x_n)$-invariant modulo $(x_2^{q^{m}},\ldots,x_n^{q^{m}})$. 
\end{lem}
\begin{proof}
	Clear.
\end{proof}	
Now, assume that the lowest $x_1$-degree of $F$ is strictly smaller than $q^m-1$ where $m \ge 1$. Write 
$$F(x_1,\ldots,x_n)=x_1^{(q-1)i}f(x_2,\ldots,x_n)+ x_1^{(q-1)(i+1)}f'(x_1,x_2,\ldots,x_n).$$ 
	\begin{lem}\label{lower-type} $f(x_2,\ldots,x_n)$ is a $q$-power of some polynomial $g(x_2,\ldots,x_n)$. Furthermore, $g(x_2,\ldots,x_n)$ is $\B(x_2,\ldots,x_n)$-invariant modulo $(x_2^{q^{m-1}},\ldots,x_n^{q^{m-1}})$. 
	\end{lem}

 \begin{proof} 
The second statement follows easily once the first is established.

For the first one, multiplying $F$ by the $\B_n$-invariant $x_1^{q^m-q-(q-1)i}$ gives a new $\B_n$-invariant modulo $I_m$:
$$F':=x_1^{q^m-q-(q-1)i}F=x_1^{q^m-q}f(x_2,\ldots,x_n)+x_1^{q^m-1}f'(x_1,\ldots,x_n).$$
Write $f(x_2,\ldots,x_n)$ as a linear combination of $x^J:=x_2^{j_2}\ldots x_n^{j_n}$. For each $2\le k\le n$, we show that $j_k$ is divisible by $q$ for all $J$.

Since $F'$ is $\B_n$-invariant modulo $I_m$, we have $$F'(x_1,\ldots,x_k+x_1,\ldots,x_n)-F'(x_1,\ldots,x_k,\ldots,x_n)=0\mod I_m.$$ 
The difference
$
x_1^{q^m-1}\big[f'(x_1, \ldots, x_k+x_1, \ldots, x_n) - f'(x_1, \ldots, x_k, \ldots, x_n)\big] 
$
is divisible by $x_1^{q^m}$, so it vanishes in the quotient ring $\Qcal_m(n)$.

On the other hand, the difference $x_1^{q^m-q}[f(x_2, \ldots, x_k+x_1, \ldots, x_n) - f(x_2, \ldots, x_k, \ldots, x_n)]$ is a linear combination of
\begin{eqnarray*}
    d(J)&:=&x_1^{q^m-q}\big[   x_2^{j_2 } \ldots (x_k+x_1)^{j_k} \ldots  x_n^{j_n}  -  x_2^{j_2 } \ldots x_k^{j_k} \ldots  x_n^{j_n} \big] \\
    &=&\sum_{s=1}^{j_k}
\binom{j_k}{s} x_1^{q^m-q+s}  x_2^{j_2} \ldots x_k^{j_k -s} \ldots  x_n^{j_n }. 
\end{eqnarray*}
It is clear that if $J\not=J'$ then $d(J)$ and $d(J')$ share no common nontrivial summand.
It follows that for each $J$, we must have $d(J)=0\mod I_m$. By inspecting the exponents of $x_1$, we see that $\binom{j_k}{s}=0$ for all $1\le s\le q-1$. This implies that $j_k$ is divisible by $q$.
\end{proof}

Recall that if
$$Y_b(I; J) :=  \delta_{1;b}^{i_1}( D_1^{j_1} \delta_{2;b}^{i_2} (  D_2^{j_2}  (   \cdots    ( \delta_{k;b}^{i_k}(D_k^{j_k})   )  \cdots  )  ) ),$$
then $\Phi Y_b(I;J)$ is defined by $$\Phi Y_b(I;J)= \delta_{2;b+1}^{i_1}( D_2^{j_1} \delta_{3;b+1}^{i_2} (  D_3^{j_2}  (   \cdots    ( \delta_{k+1;b+1}^{i_k}(D_{k+1}^{j_k})   )  \cdots  )  ) ).$$ 
The next result explains the Frobenius-like property of the $\Phi$ operator.  

\begin{lem}\label{Frob}
	If $f(x_1,\ldots,x_c)$ is a polynomial such that $f(0,x_2,\ldots,x_c)=g(x_2,\ldots,x_c)^q$ and that $\delta_{a+1;b+1}(f)$ is a polynomial, then $\delta_{a+1;b+1}(f)(0,x_2,\ldots,x_{c+1})=\big(\delta_{a;b}(g)(x_2,\ldots,x_{c+1})\big)^q.$ As a consequence, if $Y\in \Bcal_{m-1}(n-1)$, then $\Phi Y(0,x_2,\ldots,x_n)=Y(x_2,\ldots,x_n)^q.$
\end{lem}

\begin{proof}
	The first statement follows from the definition of $\delta$ as a quotient of determinants. The second then follows by induction, using the polynomiality of $Y$.
\end{proof}

We are now ready to prove the main theorem.

\begin{proof}[Proof of the spanning property for $\Bcal_m(n)$] 
Let $F$ be a homogeneous $\B_n$-invariant in $\Qcal_m(n)$ and write $F$ in the form 
	 $$F(x_1,\ldots,x_n)=x_1^{(q-1)i}f(x_2,\ldots,x_n)+ x_1^{(q-1)(i+1)}f'(x_1,x_2,\ldots,x_n).$$ 
We call $x_1^{(q-1)i}f(x_2,\ldots,x_n)$ the lowest part of $F$. If $i=[m]_q$ then it follows from Lemma \ref{top-type} that $f(x_1,\ldots,x_{n-1})$ is in $\Qcal_{m}(n-1)^{\B_{n-1}}$, and so by inductive hypothesis, $F$ is spanned by $\{\delta_{1;m}(Y)\mid Y\in \Bcal_{m}(n-1) \}$.

If $i<[m]_q$ then, by Lemma \ref{lower-type}, 
$f(x_2,\ldots,x_n)=g(x_2,\ldots,x_n)^q,$
where
$g(x_1,\ldots,x_{n-1})$ is in $\Qcal_{m-1}(n-1)^{\B_{n-1}}$. By inductive hypothesis, $g(x_2,\ldots,x_{n})$ can be written in the form
$$g(x_2,\ldots,x_n)=\sum_{Y} c_YY(x_2,\ldots,x_n)$$
where $Y\in \Bcal_{m-1}(n-1).$ Let $$F'=F-\sum_{Y} c_YD_1^i\Phi Y.$$
Then $F'$ is also an $\B_n$-invariant. By Lemma \ref{Frob}, the degree of $x_1$ in the lowest part of $F'$ is strictly greater than that of $F$, so we can repeat the above process with $F'$ to finish the proof. 
\end{proof}

\begin{proof}[Proof of the linearly independent property for $\Bcal_m(n)$]
	The first subset $\{\delta_{1;m}(Y)\mid Y\in \Bcal_{m}(n-1) \}$ of $\Bcal_{m}(n)$ is clearly linearly independent by inductive hypothesis. 
	
	It suffices to prove the linear independence of the subset $\{ D_1^a \Phi(Y) \mid Y\in \Bcal_{m-1}(n-1) \}$ for each $a<[m]_q$. 
	By Lemma \ref{Frob},  $$D_1^a \Phi(Y) =x_1^{a(q-1)}Y(x_2,\ldots,x_n)^q+ \text{monomials with $x_1$-degree $>a(q-1)$}.$$ It follows that if 
	$\sum_{Y} c_{Y} D_1^a\Phi(Y)=0$ modulo $(x_1^{q^{m}},\ldots,x_n^{q^{m}})$,  then we have $\sum_{Y} c_{Y} Y(x_2,\ldots,x_n)^q=0$ modulo $(x_2^{q^{m}},\ldots,x_n^{q^{m}})$, and so $\sum_{Y} c_{Y} Y(x_2,\ldots,x_n)=0$ modulo $(x_2^{q^{m-1}},\ldots,x_n^{q^{m-1}})$. 
By inductive hypothesis, we conclude that the coefficients $c_Y$'s vanish and the proof is complete.
\end{proof}

\begin{cor}\label{lm Yb}
In graded lexicographic order, the smallest monomial occuring in $Y_m (I, J)$ equals: 
\[\prod_{s=1}^k x_{i_1+\ldots +i_{s-1}+s}^{q^m-q^{s-1}} \cdots x_{i_1+\ldots +i_s+s-1}^{q^m-q^{s-1}} x_{i_1 + \ldots+ i_s +s}^{j_s q^{s-1}(q-1)}.\]
\end{cor}
\section{The Hilbert series}
\label{Hilbert}

In this section, we verify the Lewis, Reiner and Stanton conjecture \cite{LewisReinerStanton2017} about the Hilbert 
series $F_{n,m}(t)$ \footnote{This function should be $C_{1^n,m}(t)$ in the notation of \cite{LewisReinerStanton2017}.} of $\Qcal_m(n)^{\B_n}$. Recall the conjecture states that
\[
F_{n,m}(t)=\sum_{\beta\le 1^n, |\beta|\le m }t^{e(m,1^n,\beta)}\bmt{m\\ \beta, m-|\beta| }_{q,t},
\] 
 where
\[
e(m,1^n,\beta) = \sum_{i=1}^{n} (1-\beta_i)(q^m-q^{B_i}),   
\] 
$B_i$ denotes the partial sum $B_i = \sum_{j=1}^{i} \beta_j$, and the $(q,t)$-multinomial coefficient is given by
\[
\bmt{m\\ \beta,m-\left | \beta \right |}_{q,t} = \frac{\prod_{j=0}^{m-1} (1-t^{q^m-q^j})} {\prod_{i=1}^{n} \prod_{j=0}^{\beta_i-1} (1-t^{q^{B_i} - q^{B_{i-1} + j}}) }.
\] 
We have shown that the set $\Bcal_m (n)$ forms a basis for the $\mathbb{F}_q$-vector space $\Qcal_m (n)$, and it can be constructed inductively:  
\[
\Bcal_{m}(n)=\{\delta_{1;m}(Y)\mid  Y\in \Bcal_m(n-1) \}\sqcup
\{ D_1^a \Phi(Y) \mid  a<[m]_q, Y\in \Bcal_{m-1}(n-1) \}.\]

It is clear that $F_{n,0}(t) =1$, and by convention, we set $F_{0,m} (t)= 1$. The next result shows a similar inductive formula for $F_{n,m} (t)$. 
\begin{prop} 
For $m,n \geq 1$, $F_{n,m}(t)$ satisfies the relation:
\[
F_{n,m}(t)=t^{q^m-1}F_{n-1,m}(t)+\frac{1-t^{q^m-1}}{1-t^{q-1}}F_{n-1,m-1}(t^q).
\] 
\end{prop}

\begin{proof}
We have
\begin{eqnarray*}
	F_{n,m}(t) &=& \sum_{\beta \leq 1^n, \left | \beta \right | \leq m}{t^{e(m,1^n,\beta)} \bmt{m\\ \beta,m-\left | \beta \right |}_{q,t}}\\
	&=& \sum_{\beta=(0,{\beta}') \leq 1^n, \left | \beta \right | \leq m}{t^{e(m,1^n,\beta)} \bmt{m\\ \beta,m-\left | \beta \right |}_{q,t}} +  \sum_{\beta=(1,{\beta}') \leq 1^n, \left | \beta \right | \leq m}{t^{e(m,1^n,\beta)} \bmt{m\\ \beta,m-\left | \beta \right |}_{q,t}}.
\end{eqnarray*}
The first term in this sum can be written as follows:
\begin{eqnarray*}
	\sum_{\beta=(0,{\beta}') \leq 1^n, \left | \beta \right | \leq m}{t^{e(m,1^n,\beta)} \bmt{m\\ \beta,m-\left | \beta \right |}_{q,t}} &=& t^{q^m-1} \sum_{{\beta}' \leq 1^{n-1}, \left | {\beta}' \right | \leq m}{t^{e(m,1^{n-1},{\beta}')} \bmt{m\\ {\beta}',m-\left | {\beta}' \right |}_{q,t}}\\
	&=& t^{q^m-1}F_{n-1,m}(t).
\end{eqnarray*}
For the second term, observe that, with $\beta=(1,\beta_2,\ldots,\beta_n) = (1,{\beta}')$, we have
	$$t^{e(m,1^n,\beta)} = {\left( t^q \right)}^{e(m-1,1^{n-1},{\beta}')},$$
	and $$\bmt{m\\ \beta,m-\left | \beta \right |}_{q,t} = \frac{1-t^{q^m-1}}{1-t^{q-1}} \bmt{m-1\\ {\beta}',m-1-\left | {\beta}' \right |}_{q,t^q}.$$
Hence the second term can be rewritten as:
\begin{eqnarray*}
	\sum_{\beta=(1,{\beta}') \leq 1^n, \left | \beta \right | \leq m}{t^{e(m,1^n,\beta)} \bmt{m\\ \beta,m-\left | \beta \right |}_{q,t}} &=& \frac{1-t^{q^m-1}}{1-t^{q-1}}  \sum_{{\beta}' \leq 1^{n-1}, \left | \beta \right | \leq m-1}{(t^q)^{e(m-1,1^{n-1},{\beta}')} \bmt{m-1\\ {\beta}',m-1-\left | {\beta}' \right |}_{q,t^q}} \\
	&=&\frac{1-t^{q^m-1}}{1-t^{q-1}} F_{n-1,m-1}(t^q).
\end{eqnarray*}
The result follows.
\end{proof}

\begin{proof}[Proof of Proposition \ref{B(m,n) description}]
In the remaining part of this section, we prove Proposition \ref{B(m,n) description} about the decomposition of $\Bcal_m(n)$:
\[
\Bcal_{m}(n)=\coprod_{k=1}^{\min(n,m+1)}\Bcal_m^k(n),
\]
where $\Bcal_m^k(n)$ denotes the set consisting of all the elements $Y_m(I; J)$ for which the sequences $I=(i_1,\ldots,i_k)$ and $J=(j_1,\ldots,j_k)$ satisfy the following conditions:
	$$\begin{cases}
		i_1+\cdots+i_k=n-k,\\ 
		j_1 < [m]_q, 
		\dots,
		j_{k-1}< [m-k+2]_q,
		j_k\le [m-k+1]_q.
	\end{cases}$$
 First of all, it is easy to see that $\delta_{1;m} \Bcal_{m} (n-1)$ corresponds to the set $Y_m (I, J)$ such that 
	\[ 
	  \begin{cases}
	 	i_1+\cdots+i_k=n-k,\\
	 	i_1 \geq 1,\\ 
	 	j_1 < [m]_q, 
	 	\cdots,
	 	j_{k-1}< [m-k+2]_q,
	 	j_k\le [m-k+1]_q.
	 \end{cases}
	\]
	Similarly, the set $\{D_1^a \Phi Y \mid a <[m]_q, Y \in \Bcal_{m-1}(n-1)\}$ corresponds to those $Y_m (I,J)$ where $i_1=0$.
	The proof of this is straightforward (but tedious), and the two cases $n \leq m+1$ or $n > m+1$ need to be considered separately. We leave the details for the interested reader.   
	\end{proof}

 \begin{rmk}\label{rmk: summand of Fnm}
Suppose $\beta = (\beta_1, \ldots, \beta_n)$ is a weak composition of non-negative integers such that $\beta \leq 1^n$ and $\left | \beta \right | \leq m$. The summand $t^{e(m,1^n, \beta)} \bmt{m\\ \beta,m-\left | \beta \right |}_{q,t}$ of $F_{n,m} (t)$ can be described as follows: It is the Hilbert series for the subspace of $\Qcal_m (n)^{\B_n}$  spanned by those $Y_m (I, J)$s where the pairs $(I, J)$ satisfy the following conditions: 
\begin{enumerate}
    \item $\beta_i =1$ in positions $i_1 + \ldots +i_s +s$, for $1 \leq s \leq k-1$.
    \item When $s=k$, $i_1 + \ldots + i_k +k =n$, then $\beta_n = 0$ if $j_k = [m-k+1]_q$ and $\beta_n =1$ if $j_k < [m-k+1]_q$. 
    \item $\beta_i = 0$ in all other positions. 
\end{enumerate}
 \end{rmk}

\section{A bijection} \label{bijection}
 Lewis, Reiner and Stanton arrived at the parabolic conjectures by comparing the spaces of invariants of two Brauer isomorphic $\GL_n (\mathbb{F}_q)$-representations, namely the graded $\Qcal_m (n)$ and the ungraded parking space $\mathbb{F}_q [\mathbb{F}_{q^m}^n]$. In this section, we will construct a bijection between the $\B_n$-invariant subspaces
\[
\Qcal_{m}(n)^{\B_n} \to (\mathbb{F}_q [\mathbb{F}_{q^m}^n])^{\B_n}, 
\]or rather a set bijection from our basis $\Bcal_m (n)$ to another natural basis 
for $(\mathbb{F}_q [\mathbb{F}_{q^m}^n])^{\B_n}$ that we are going to describe. The bijective map indicates that in a sense, our basis for the $\mathbb{F}_q$-vector space $\Qcal_m(n)$ is a natural one. 

We begin by recalling the description in terms of flags of the $\B_n$-invariant subspace of $\mathbb{F}_q [\mathbb{F}_{q^m}^n]$ as analyzed in \cite[Section 6]{LewisReinerStanton2017}. Since $\mathbb{F}_q [\mathbb{F}_{q^m}^n]$ admits a permutation action of $\GL_n$, its $\B_n$-invariants can be identified with its $\B_n$-orbits. A $(\beta, m - \left|\beta\right|)$-flag in $\mathbb{F}_{q^m}$ is a nested sequence of $\mathbb{F}_q$-subspaces:  
\[
\{0\} = V_0 \subset V_1 \subset \ldots \subset V_{n} \subset \mathbb{F}_{q^m}
\]
such that $\dim V_{\ell} =  \sum_{j=1}^{\ell} \beta_j$. Let $X_{m}(n)$ denote the set of all $(\beta, m - \left|\beta\right|)$-flags where $\beta$ is a weak composition satisfying the following $2$ conditions:
\[
\beta \leq 1^n \quad \text{and} \quad \left|\beta\right| \leq m. 
\]
The first condition implies $\beta_i \leq 1$ for all $i$. Thus, the dimension jump of the spaces, except possibly at $V_n$, in a $(\beta, m-\left|\beta\right|)$-flag is at most $1$. We will call an element of $X_{m}(n)$ a $\B_n$-flag. It was proved that the set $X_m(n)$ naturally indexes the $\B_n$-orbits of $\mathbb{F}_{q^m}^n$, and hence forms a basis of $(\mathbb{F}_q [\mathbb{F}_{q^m}^n])^{\B_n}$. 

We first consider the case $n=1$. In this case, $X_m(1)$ consists of 2 families: The first family, corresponding to $\beta = (0)$, which has only one flag - the trivial flag $V_0 = V_1 = \{0\}$. The second family, which corresponds to $\beta = (1)$ consists of all sequences $\{0\} = V_0 \subset V_1 \subset \mathbb{F}_{q^m}$ where $\dim V_1 = 1$. 

We construct a bijection $\Bcal_m (1) \to X_{m}(1)$ by sending $Y_m (0, [m]_q)$ to the trivial flag, and $Y_m (0, j_1)$ with $0 \leq j_1 < [m]_q$ bijectively to the 
set of $[m]_q$ lines in $\mathbb{F}_{q^m}$. 
  
When $n>1$, for each fixed weak composition $\beta$ with $\beta \leq 1^n$ and $\left | \beta \right | \leq m$, we construct a bijection from the set $Y_{\beta}$ of $(\beta, m-\left | \beta \right |)$-flags to the set of polynomials $Y_m (I,J)$ where $(I, J)$ satisfies the conditions given in Remark \ref{rmk: summand of Fnm} as follows: The first space $V_1$ of dimension $1$ in $\mathbb{F}_{q^m}$ corresponds bijectively to the set $0 \leq j_1 < [m]_q$. A 2-dimensional space containing $V_1$ corresponds to lines in a $(m-1)$-dimensional vector space, which corresponds bijectively to the set $0 \leq j_2 <[m-1]_q$, etc.

\section{Conjectures}
The Lewis-Reiner-Stanton conjectures provide a description of the Hilbert series of the invariant ring of the truncated polynomial algebra under the action of parabolic subgroups. In this section, we offer a refinement of their conjectures by proposing an explicit $\mathbb{F}_q$-basis for the invariant ring $\Qcal_m (n)^{\Par_{n_1,\ldots,n_k}}$ where $\Par_{n_1,\ldots,n_k}$ is a parabolic subgroup of $\GL_n$. In this section, $m$ is fixed, and we write $\delta_a$ for $\delta_{a;m}$.

 	We begin with a result describing how to obtain higher invariants from lower ones.  
		\begin{prop}\label{GLs to GLn}
		If $0 \leq s\leq n$ and $f(x_1,\ldots,x_s)$ is a $\GL_s$-invariant polynomial, then the polynomial $\delta^{n-s}_{s+1}(f)$ is $\GL_n$-invariant modulo $(x_1^{q^m},\ldots,x_n^{q^m}).$
	\end{prop}
	\begin{proof}Since $\delta^{n-s}_{s+1}(f)$ is symmetric by Proposition \ref{iterate-delta}, it is sufficient to show that $\delta^{n-s}_{s+1}(f)$ is $(n,m)$-invariant. But since $f$ is a polynomial which is $\GL_s$-invariant, 
 $\delta^{n-s}_{s+1}(f)$ is a polynomial by Proposition \ref{delta-poly}. Finally it is $(n,m)$-invariant by Proposition \ref{k-invariant}.
	\end{proof}
This Proposition indicates that we may be able to obtain $\GL_n$-invariants from $\GL_s$-invariants by applying appropriate iteration of the operator $\delta_{s+1}$. Note that $\delta_{s+1}^{n-s}$ raises the degree by $(n-s)(q^m-q^s)$.     

 Lewis, Reiner and Stanton \cite{LewisReinerStanton2017} conjectured that the Hilbert series of $\Qcal_m(n)^{\GL_n}$ is given by the following formula: 
\begin{equation}\label{Hnm} 
C_{n,m}(t)=\sum_{s=0}^{\min(m,n)}t^{(n-s)(q^m-q^s)}\bmt{m\\s}_{q,t}, 
\end{equation}
where the $(q,t)$-binomial coefficient $\bmt{m\\s}_{q,t}$ is given by
\[
\bmt{m\\s}_{q,t}=\prod_{i=0}^{s-1}\frac{1-t^{q^m-q^i}}{1-t^{q^s-q^i}}.
\] 
This further generalization of the classical binomial coefficients was introduced and studied extensively in \cite{Reiner-Stanton2010}. Formula \eqref{Hnm} for the conjectural Hilbert series $C_{n,m}(t)$ and Proposition \ref{GLs to GLn} suggests that in order to obtain $\Qcal_m(n)^{\GL_n}$, one needs only apply the operator $\delta_{s+1}^{n-s}$ on a part of the Dickson algebra $\mathcal{D}_s$ of $\GL_s$-invariants whose Hilbert series equals $\bmt{m\\s}_{q,t}$. 

It turns out that a candidate for such a piece of the Dickson algebra was already discussed in Section 5 of \cite{LewisReinerStanton2017}. We review their discussion here. A typical element in $\mathcal{D}_s$ is of the form 
\[
Q_{s,s-1}^{m_1} \ldots Q_{s,s-i}^{m_i} \ldots Q_{s,s-s}^{m_s},  
\]
and can be associated in an obvious way with a partition $\mu$ whose part sizes are restricted to the set $\{q^s - q^{s-1}, q^s-q^{s-2}, \ldots, q^s-q^0\}$. Such a partition $\mu$ is said to be $q$-compatible with another partition $\lambda$ having at most $s$ nonzero parts ($\lambda_{s+1}=0$) if for all $1 \leq i \leq s$, 
\[\frac{q^{\lambda_i}- q^{\lambda_{i+1}}}{q-1} \leq m_i < \frac{q^{\lambda_i}- 
q^{\lambda_{i+1}}}{q-1} + q^{\lambda_i} = \frac{q^{\lambda_i+1}-q^{\lambda_{i+1}}}{q-1}.
\]
Each partition $\mu$ associated to a Dickson monomial $Q_{s,s-1}^{m_1} \ldots Q_{s,0}^{m_s}$ is $q$-compatible with at most one partition $\lambda$ with at most $s$ nonzero parts. Let $\Delta^m_s$ denote the set of all partitions $\mu$ which are $q$-compatible with a $\lambda$ that fits inside an $s \times (m-s)$ rectangle (that is, $\lambda$ has at most $s$ nonzero parts and $\lambda_1 \leq m-s$). In terms of Dickson monomials, $\Delta^m_s$ is explicitly given by 
\[
\Delta^m_s=\coprod_{m-s \geq \lambda_1 \geq \ldots \geq \lambda_s \geq \lambda_{s+1}=0 } \bigl\{ Q_{s,s-1}^{m_1} \ldots Q_{s,0}^{m_s} \mid  \frac{q^{\lambda_i}- q^{\lambda_{i+1}}}{q-1} \leq m_i < \frac{q^{\lambda_i+1}- q^{\lambda_{i+1}}}{q-1}, i=1,\ldots,s \bigr\}.
\]
 It follows from \cite[Proposition 5.6]{Reiner-Stanton2010} that the Hilbert series of the set  is exactly $\bmt{m\\s}_{q,t}$. 
We now describe another candidate, which seems more convenient to work with, for a part of the Dickson algebra whose Hilbert series is $\bmt{m\\s}_{q,t}$. 
For each partition $\lambda = (\lambda_1 \geq \ldots \geq \lambda_s \geq 0)$, Macdonald \cite{Macdonald-Schur} defines a variation of Schur functions, now called the 7th variation, which are $\GL_s$-invariant as follows: 
\[
S_\lambda=\frac{\det(x_i^{q^{\lambda_j +s-j }})_{1\le i,j\le s} }{\det(x_i^{q^{s-j}})_{1\le i,j\le s}}.
\]
The Dickson invariant $Q_{s,i}$ is the Schur function for the partition $1^{s-i}$ and $\delta_{a;b}(1)$ is a special case of this variation. Note also that $S_{\lambda}$ can be constructed inductively using the $\delta$ operator, $$S_\lambda=\delta_{s;\lambda_1+s-1}(S_{\lambda_2,\ldots,\lambda_s}).$$

Let $\nabla^m_s$ denote the disjoint union of 
\[
\bigl\{S_\lambda Q_{s,s-1}^{a_1}\cdots Q_{s,0}^{a_s}\mid 0\le a_i<q^{\lambda_i}, i = 1,\ldots, s \bigr\},
\]
where, as in the case of $\Delta^m_s$, $\lambda$ runs over the set of all partitions that can fit inside an $s \times (m-s)$ rectangle. It is possible to check that $\Delta^m_s$ and $\nabla^m_s$ have the same Hilbert series.

The above discussion led us to propose the following refinement of the Lewis, Reiner and Stanton conjecture for the full general linear group:   
	
	\begin{conj}
For $m, n \geq 1$, the set	$\{ \delta^{n-s}_{s+1}(f)\mid f\in\nabla^m_s, 0\le s\le \min(m,n)  \}$ is a basis for the $\mathbb{F}_q$-vector space $\Qcal_m(n)^{\GL_n}$.
	\end{conj}
	It is clear that if this conjecture is true, then the $\GL_n$-invariant subspace of $\Qcal_m (n)$ has the right Hilbert series predicted by Lewis, Reiner and Stanton. 
	\begin{example} \label{GL2-inv}
	Let us consider the case $m \geq n = 2$. Our conjecture says that the following 3 families form a basis for $\Qcal_m(2)^{\GL_2}$. 
	\begin{enumerate}
		\item $\delta_1^2 (\nabla^m_0)$: Since $\nabla^m_0 = \{1\}$, this family contains only one polynomial, which is the top degree element $x_1^{q^m-1} x_2^{q^m-1}$. 
		\item $\delta_2^1 (\nabla^m_1)$: Unraveling the definition, we see that $\nabla^m_1$ consists of Dickson polynomials $Q_{1,0}^{k}$ with $k < [m]_q$. This family $\delta_2 (Q_{1,0}^{k})$ is exactly the $\GL_2$-invariants denoted $y_k$ in Goyal \cite[Proposition 4.1]{Goyal18}.
		\item $\delta_3^0 (\nabla^m_2)$: This is just $\nabla^m_2$ - part of the Dickson algebra that survives after taking modulo $(x_1^{q^m},x_2^{q^m})$. 
	\end{enumerate}
	\end{example}

In general, denote by $\P_{n,s}$ the iterated $\delta^{n-s}_{s+1}$, we may represent the conjectural picture as follows:
	\[ 
	\xymatrix{
		\nabla^m_0\ni f \ar[r]^-{\delta_{1}} &  \P_{1,0}(f) \ar[r]^-{\delta_{1}} & \P_{2,0}(f) \ar[r]^-{\delta_{1}} & \P_{3,0}(f) \ar[r]^-{\delta_{1}} &\cdots \\
		& 	\nabla^m_1\ni f \ar[r]^{\delta_{2}} & \P_{2,1}(f)\ar[r]^-{\delta_{2}} & \P_{3,1}(f) \ar[r]^-{\delta_{2}}& \cdots \\
		& &	\nabla^m_2\ni  f \ar[r]^-{\delta_{3}}& \P_{3,2}(f) \ar[r]^-{\delta_{3}} &\cdots \\
		& & & 	\nabla^m_3\ni f  \ar[r]^-{\delta_{4}} & \cdots 
	}
	\]
	The $n$th column consisting of $\P_{n,s}$, $0 \leq s \leq n$, would give a basis for $\Qcal_m(n)^{\GL_n}$. The horizontal lines are disconnected in the following sense. 
	\begin{prop} 
	For $s \geq 0$ and $f(x_1,\ldots,x_s)$ a rational function, we have $\delta_{s+2} \delta_{s+1}(f)=0$. 
	\end{prop}
	\begin{proof}
	 Applying Proposition \ref{composite-delta} for $h=k=1$, $r=s+1$, and $g=1$ - the constant function, we get 
		\[
		\delta_{s+2} \delta_{s+1} (f) = \sum_{T \in \Tcal(s,1)} A_{s+1;T} (1) \delta_{s+1}^2 (\beta_T f). 
		\]
		We will show that $A_{r;T} (1)= 0$. Indeed, note that the set $\Tcal (s,1)$ consists of exactly $s+1$ elements $T_j$, $0 \leq j \leq s$, corresponding to the composition of $1$ in which $t_j =1$ and all other $t_{\neq j} = 0$. 
		
		By definition, we have
		\[
		A_{r;T_j} (1) = \sum_{I \subset [r+1], |I|=1} \frac{\alpha_{T_j} (I)}{V(\bar I, I)} = \sum_{i=1}^{r+1} \frac{x_i^{q^j}}{V(x_1 \ldots, \hat{x_i}, \ldots, x_{r+1}, x_i)}. 
		\]
	The last sum can be rewritten as
	\[
\frac{	\sum_{i=1}^{r+1} (-1)^{r+1-i} x_i^{q^j} L(x_1 \ldots, \hat{x_i}, \ldots, x_{r+1})}{L(x_1, \ldots, x_{r+1})}.
	\]
	The numerator of this quotient is just the Laplace expansion along the last row of a $(r+1) \times (r+1)$ matrix whose $(u,v)$-entry is $x_v^{q^{u-1}}$ if $1 \leq u \leq r$ and equals $x_v^{q^j}$ if $u=r+1$. Since $j \leq s$, this matrix has 2 rows that are the same.   
\end{proof}

For the parabolic subgroup $\Par_{n_1,\ldots,n_k}$ of $\GL_n$ corresponding to the composition $n=n_1 + \ldots +n_k$, we propose the following variant: 
\begin{conj} 
A basis for $\Qcal_m(n)^{\Par_{n_1,\ldots,n_k}}$ is given inductively by
	$$\Bcal_m(n_1,\ldots,n_k)=\{\delta_{s+1;m}^{n_1-s}(f\cdot  \Phi^{s}g)\mid 0\le s \le \min(n_1,m), f\in \nabla^m_{s}, g \in \Bcal_{m-s}(n_2,\ldots,n_k)\}.$$
	Here the operator $\Phi$ is given by the ``rule": $\Phi(\delta_{a;b})=\delta_{a+1;b+1}$, $\Phi(Q_{r,i})=Q_{r+1,i+1}$ and $\Phi (S_\lambda)=S_{\lambda,0}$.
\end{conj}
\begin{example}
 Let $n=3$ and consider the composition $3=1+2$. For simplicity, assume that $m \geq 3$. Then the above conjecture predicts that a basis for $\Qcal_m (3)^{\Par_{1,2}}$ is the set $\Bcal_m (1,2)$ consisting of the following families: 
 \begin{enumerate}
     \item $\delta_{1;m} (g)$, where $g \in \Bcal_{m} (2)$, 
     \item $f \cdot \Phi g$, where $f \in \nabla^m_1$, $g \in \Bcal_{m-1} (2)$. 
 \end{enumerate}
The set $\Bcal_m(2)$ was already described in Example \ref{GL2-inv}. 
\end{example}

\begin{rmk} In \cite{Ha-Hai-Nghia-2-2023}, the above conjectures are verified for all parabolic subgroups in the case $n\le 3$. 
\end{rmk}

\bibliographystyle{amsplain}

\providecommand{\bysame}{\leavevmode\hbox to3em{\hrulefill}\thinspace}
\providecommand{\MR}{\relax\ifhmode\unskip\space\fi MR }
\providecommand{\MRhref}[2]{%
	\href{http://www.ams.org/mathscinet-getitem?mr=#1}{#2}
}
\providecommand{\href}[2]{#2}

\end{document}